\documentclass[11pt]{amsart}
\usepackage{latexsym,amssymb,amsfonts,amsmath,amsthm,graphicx,color,fullpage}
\usepackage[vcentermath,enableskew]{youngtab}
\usepackage[all]{xy}
\usepackage{verbatim}
\usepackage{accents}
\usepackage[colorlinks=true,citecolor=black,linkcolor=black,urlcolor=blue]{hyperref}

\newcommand{\arxiv}[1]{\href{http://arxiv.org/abs/#1}{\texttt{arXiv:#1}}}

\theoremstyle{plain}
\newtheorem{theorem}{Theorem}[section]
\newtheorem{proposition}[theorem]{Proposition}
\newtheorem{lemma}[theorem]{Lemma}
\newtheorem{mainlemma}{Lemma 6.2}
\newtheorem{maintheorem}{Theorem 6.7}
\newtheorem{corollary}[theorem]{Corollary}
\theoremstyle{definition}
\newtheorem{definition}[theorem]{Definition}

\theoremstyle{remark}

\numberwithin{equation}{section}

\DeclareMathOperator*{\row}{Row}
\DeclareMathOperator*{\aposet}{\mathbf{A}}
\DeclareMathOperator*{\gyr}{Gyr}
\DeclareMathOperator*{\togstat}{\mathfrak{T}}

\author[J.\ Striker]{Jessica Striker}
\thanks{The author  is supported in part by North Dakota EPSCoR and a grant from the NSA}

\address{Department of Mathematics,
North Dakota State University Dept \# 2750,
PO Box 6050,
Fargo, ND 58108-6050
USA}
\email{jessica.striker@ndsu.edu}

\title{The toggle group, homomesy, \linebreak and the Razumov-Stroganov correspondence}
\begin{document}

\begin{abstract}
The Razumov-Stroganov correspondence, an important link between statistical physics and combinatorics proved in 2011 by L.\ Cantini and A.\ Sportiello, relates the ground state eigenvector of the $O(1)$ dense loop model on a semi-infinite cylinder to a refined enumeration of fully-packed loops, which are in bijection with alternating sign matrices. 
This paper reformulates a key component of this proof in terms of posets, the toggle group, and homomesy, and proves two new homomesy results on general posets which we hope will have broader implications. 
\end{abstract}

\maketitle

\section{Introduction}
The Razumov-Stroganov conjecture~\cite{razstrog}, proved in 2011 by L.~Cantini and A.~Sportiello~\cite{razstrogpf} who began calling it the \emph{Razumov-Stroganov correspondence}~\cite{razstrogpf2}, provides an important bridge between the statistical physics of the $O(1)$ dense loop model on a semi-infinite cylinder and combinatorial properties of \emph{alternating sign matrices}. The conjecture remained open for several years and intrigued many combinatorialists and physicists, motivating much cross-disciplinary work. Cantini and Sportiello's first proof~\cite{razstrogpf} (which proved not only the conjecture but also several generalizations) relied on a detailed analysis of the action of \emph{gyration} on \emph{fully-packed loops}, objects in bijection with alternating sign matrices. 
 
This paper relates Cantini and Sportiello's first proof~\cite{razstrogpf} of the Razumov-Stroganov correspondence to both the \emph{toggle group} 
and the \emph{homomesy phenomenon}, which we describe in the following paragraphs. (Cantini and Sportiello also published a one-parameter refinement~\cite{razstrogpf2} of the Razumov-Stroganov correspondence in 2014, in which they also proved some additional generalizations; the present paper addresses only~\cite{razstrogpf}, but the viewpoint described here could be a fruitful perspective on~\cite{razstrogpf2} as well.) Since this paper ties together several disparate topics, we aim to explain these concepts and connections in a manner accessible to readers in a variety of fields.

The \emph{toggle group} is a subgroup of the symmetric group generated by certain involutions, called \emph{toggles}, which acts on \emph{order ideals} of a partially ordered set, or \emph{poset}. The toggle group was defined in a 1995 paper by P.\ Cameron and D.~Fon-der-Flaass~\cite{fonderflaass} and further studied in a 2012 paper by the author with N.~Williams~\cite{prorow}. In that paper, N. Williams and the author investigated toggle group actions on many combinatorial objects that can be encoded as order ideals of a poset, such as subsets, partitions, Catalan objects, certain plane partitions, and alternating sign matrices. In many cases, known actions on these objects could be expressed as toggle group actions~\cite{prorow}. The authors also used conjugacy of toggle group elements to show that different combinatorially meaningful actions have the same orbit structure, in many cases giving easy proofs of the \emph{cyclic sieving phenomenon} of V.\ Reiner, D.\ Stanton, and D.\ White~\cite{CSP}. The relevant result from~\cite{prorow} for the purpose of this paper we list 
as Proposition~\ref{prop:asmgyr}, namely, that gyration on fully-packed loops can be expressed as a toggle group action. 

The \emph{homomesy phenomenon} was isolated in 2013 by J.~Propp and T.~Roby~\cite{homomesy}. A statistic on a set of combinatorial objects exhibits \emph{homomesy} (or is \emph{homomesic}) with respect to a bijective action if the average value of the statistic on any orbit of the action is equal to the global average value of the statistic. That is, the orbit-average of the statistic is independent of the chosen orbit. In~\cite{homomesy}, Propp and Roby investigated homomesy of toggle group actions in certain posets and also noted that homomesy occurs in many other contexts; this has been a growing area of research~\cite{bloom_homomesy}~\cite{EinsteinPropp} \cite{SHaddadan}~\cite{HopkinsZhang}~\cite{Rush}.

Beyond relating the above concepts of homomesy and the toggle group to the \linebreak Razumov-Stroganov correspondence, the main new contributions of this paper are two homomesy results on general  posets which we hope will have broader implications. Both results concern the \emph{toggleability} statistic $\togstat_p$, which we introduce in Definition~\ref{def:toggleability}. 

The first new result involves
\emph{rowmotion}, a ubiquitous action studied in~\cite{fonderflaass} and~\cite{prorow}. 

\begin{mainlemma}
Given any finite poset $P$ and $p\in P$, $\togstat_p$ is homomesic with average value~$0$ with respect to rowmotion.
\end{mainlemma}

For the other new result, in Definition~\ref{def:gyr} we define a toggle group action $\gyr$ on the order ideals of any finite ranked poset. 

\begin{maintheorem}
Given any finite ranked poset $P$ and $p\in P$, 
the toggleability statistic $\togstat_p$ 
is homomesic with average value $0$ with respect to the toggle group action $\gyr$.
\end{maintheorem}

In \cite{prorow}, N. Williams and the author proved that $\gyr$ on the alternating sign matrix poset reduces to Wieland's gyration on fully-packed loops. Thus, when specialized to the case of the alternating sign matrix poset, Theorem~\ref{thm:genposet} yields Lemma 4.1 from Cantini and Sportiello's proof of the Razumov-Stroganov conjecture~\cite{razstrogpf}. We list this result as Corollary~\ref{cor:asmcor}. We hope that interpreting problems of this sort in terms of posets and homomesy may provide a more general setting within which to investigate more correspondences of Razumov-Stroganov type which are, as yet, unproved.

The paper is organized as follows. Section~\ref{sec:toggle} gives background on posets, the toggle group, and rowmotion. Section~\ref{sec:asmfplgyr} gives definitions of and poset interpretations for alternating sign matrices, fully-packed loops, and gyration. Section~\ref{sec:rs} states and explains the Razumov-Stroganov correspondence. Section~\ref{sec:homomesy} gives background on homomesy. 
Section~\ref{sec:results} defines the toggle group action $\gyr$, proves two homomesy results (Lemma~\ref{lem:rowhomomesy} and Theorem~\ref{thm:genposet}), and applies Theorem~\ref{thm:genposet} to Cantini and Sportiello's proof of the Razumov-Stroganov conjecture. 

\section{Posets, the toggle group, and rowmotion}
\label{sec:toggle}
In this section, we first give the relevant standard definitions on posets and distributive lattices, following Chapter 3 of~\cite{Stanley_1}. We then give definitions related to the more recent notions of rowmotion and the toggle group.

A \emph{partially ordered set} or \emph{poset} $P$ is a set with a partial order ``$\leq$'' that is reflexive, antisymmetric, and transitive. Given $p,q\in P$, $p$ \emph{covers} $q$ if $q\leq p$, $q\neq p$, and if $q\leq r\leq p$ then $r=q$ or $r=p$. The \emph{Hasse diagram} of $P$ is the graph in which the elements of $P$ are vertices and the edges are the covering relations, drawn such that the larger element is ``above'' the smaller in each cover.  The \emph{dual poset} of $P$ is the poset given by reversing the partial order. An \emph{order ideal} of $P$ is a subset $I\subseteq P$ such that if $p\in I$ and $q\leq p$ then $q\in I$. Denote as $J(P)$ the set of order ideals of $P$. We draw an order ideal on the Hasse diagram such that each element in the order ideal is represented by a filled-in circle ``$\bullet$'' and all other elements are represented by an open circle ``$\circ$''.

A poset $P$ is a \emph{lattice} if for any $p,q\in P$, there is a well-defined \emph{meet} (greatest lower bound), denoted $p\wedge q$, and a well-defined \emph{join} (least upper bound), denoted $p\vee q$. A lattice is \emph{distributive} if the meet and join operations distribute over each other. It is easy to see that the set of order ideals $J(P)$ has the structure of a distributive lattice with meet and join given by set union and intersection. The \emph{fundamental theorem of finite distributive lattices} says that the converse also holds: every finite distributive lattice $L$ is isomorphic to the lattice of order ideals $J(P)$ for some poset $P$. Given a finite distributive lattice $L$, a poset $P$ such that $L\cong J(P)$ is composed of the \emph{join irreducible} elements of $L$, that is, the elements $x\in L$ such that if $x=p\vee q$, then $x=p$ or $x=q$. The partial order on $P$ is given by the induced order from $L$. See Sections 3.3 and 3.4 of~\cite{Stanley_1}.

An \emph{antichain} of $P$ is a subset $A\subseteq P$ of mutually incomparable elements. 
Given any finite poset, there are two natural bijections from order ideals to antichains. Given a subset $X\subseteq P$, let $X_{\max}$ denote the maximal elements of $X$ and $X_{\min}$ denote the minimal elements of $X$. Given an order ideal $I\in J(P)$, one bijection to antichains is $I\to I_{\max}$ and another is $I\to (P\setminus I)_{\min}$. Given an antichain $A$, the order ideal $I$ such that $A=I_{\max}$ is the order ideal \emph{generated} by $A$. See Section 3.1 of~\cite{Stanley_1}.

We use the bijections of the previous paragraph between order ideals and antichains to define an interesting action, studied under various names by many authors. See~\cite{prorow} for a detailed history.
\begin{definition}
\label{def:row}
\emph{Rowmotion} is a bijective action on order ideals defined as the following composition. Given $I\in J(P)$, let $\row(I)$ be the order ideal generated by the antichain $(P\setminus I)_{\min}$.
\end{definition}

%That is, given an order ideal $I\in J(P)$, $\row(I)$ is the order ideal generated by the minimal elements of $P\setminus I$. 
We discuss in Theorem~\ref{thm:linextrow} below a different characterization of rowmotion using the toggle group. The toggle group characterization is often easier to work with than Definition~\ref{def:row}, but in Lemma~\ref{lem:rowhomomesy}, we prove a new homomesy result for rowmotion which follows directly from Definition~\ref{def:row}.

In~\cite{fonderflaass}, P.\ Cameron and D.\ Fon-der-Flaass defined a group acting on $J(P)$, dubbed the \emph{toggle group} in~\cite{prorow}.

\begin{definition}
For each $p \in P$, define $t_p: J(P) \to J(P)$ to act by \emph{toggling} $p$ if possible.  That is, if $I \in J(P)$,
$$t_p (I) = \left\{
	\begin{array}{ll}
		I \cup \{p\} & \text{ if } p\in (P\setminus I)_{\min},\\
		I\setminus \{p\} & \text{ if } p \in I_{\max},\\
		I & \text{ otherwise}.\\
	\end{array} \right.
$$
\end{definition}

\begin{definition}[\cite{fonderflaass}]
	The \emph{toggle group} $T(P)$ of a poset $P$ is the subgroup of the symmetric group $\mathfrak{S}_{J(P)}$ on all order ideals generated by the toggles $\{ t_p\}_{p\in P}$.
\end{definition}

Note that toggles have the following commutation condition.
\begin{lemma}[\cite{fonderflaass}]
\label{lem:commute}
$t_p$ and  $t_q$ commute if and only if there is no covering relation between $p$ and $q$.
\end{lemma}
 
Cameron and Fon-der-Flaass characterized rowmotion as an element of $T(P)$.

\begin{theorem}[\cite{fonderflaass}]
\label{thm:linextrow}
Given any finite poset $P$, $\row$ acting on $J(P)$ is equivalent to the toggle group element given by toggling the elements of $P$ in the reverse order of any linear extension (that is, from top to bottom).
\end{theorem}

In~\cite{prorow}, N.\ Williams and the author used the toggle group to show that rowmotion has the same orbit structure as other combinatorially meaningful actions, 
in many cases giving easy proofs of the \emph{cyclic sieving phenomenon} of V.\ Reiner, D.\ Stanton, and D.\ White~\cite{CSP}. The main tool the authors used to prove these results was the conjugacy of toggle group elements proved in the following theorem, since conjugate permutations have the same orbit structure.

\begin{theorem}[\cite{prorow}]
Given a finite ranked poset $P$ with a suitable definition of columns, rowmotion (toggle top to bottom by row) and promotion (toggle left to right by column) are conjugate elements in the toggle group.
\end{theorem}

In the next section, we recall the definitions of alternating sign matrices, fully-packed loops, and gyration, and review their interpretation from the perspective of posets.

\section{A poset perspective on alternating sign matrices, fully-packed loops, and gyration}

In this section, we define alternating sign matrices and other related concepts, paying special attention to their poset interpretations.

\label{sec:asmfplgyr}
\begin{definition}
Alternating sign matrices (ASMs) are square matrices with the following properties:
\begin{itemize}
\item
entries $\in \{0,1,-1\}$,
\item
the entries in each row/column sum to 1, and
\item
the nonzero entries in each row/column alternate in sign.
\end{itemize}
The ASMs with no $-1$'s are the permutation matrices.
\end{definition}

\begin{figure}[h]
\begin{center}
\scalebox{.85}{
$\left( 
\begin{array}{rrr}
1 & 0 & 0 \\
0 & 1 & 0\\
0 & 0 & 1
\end{array} \right)
\left( 
\begin{array}{rrr}
1 & 0 & 0 \\
0 & 0 & 1\\
0 & 1 & 0
\end{array} \right)
\left( 
\begin{array}{rrr}
0 & 1 & 0 \\
1 & 0 & 0\\
0 & 0 & 1
\end{array} \right)
\left( 
\begin{array}{rrr}
0 & 1 & 0 \\
1 & -1 & 1\\
0 & 1 & 0
\end{array} \right)
\left( 
\begin{array}{rrr}
0 & 1 & 0 \\
0 & 0 & 1\\
1 & 0 & 0
\end{array} \right)
\left( 
\begin{array}{rrr}
0 & 0 & 1 \\
1 & 0 & 0\\
0 & 1 & 0
\end{array} \right)
\left( 
\begin{array}{rrr}
0 & 0 & 1 \\
0 & 1 & 0\\
1 & 0 & 0
\end{array} \right)
$
}
\end{center}
\caption{The seven $3\times 3$ alternating sign matrices.}
\label{ex:asm3x3}
\end{figure}

We recall the poset interpretation of ASMs, first introduced in \cite{ELKP_DOMINO1}  and further studied in~\cite{TREILLIS} and \cite{StrikerPoset}. This partial order can be naturally and equivalently defined using any of the following objects in bijection with ASMs: monotone triangles, corner sum matrices, or height functions. Here we give the definition using height functions.
(See~\cite{ProppManyFaces} for further details on the bijections between these objects.)

\begin{definition}
\label{def:heightfcn}
	A \emph{height function matrix} of order $n$ is an $(n+1)\times (n+1)$ matrix $(h_{i,j})_{0\leq i,j \leq n}$ with $h_{0,k} = h_{k,0} = k$ and $h_{n,k} = h_{k,n} = n-k$ for $0 \leq k \leq n$, and such that adjacent entries in any row or column differ by $1$.
\end{definition}

A bijection between $n\times n$ ASMs and height function matrices of order $n$ is given
by mapping an ASM $(a_{i,j})_{1\leq i,j\leq n}$ to the height function matrix 
\[\left(h_{i,j}\right)_{0\leq i,j\leq n}=\left( i + j - 2 \sum_{i'=1}^{i} \sum_{j'=1}^{j} a_{i',j'}\right)_{0\leq i,j\leq n}.\]

Define a partial ordering on $n\times n$ ASMs by componentwise comparison of the corresponding height function matrices. This poset is a self-dual distributive lattice which is the \emph{MacNeille completion} of the \emph{Bruhat order} on the symmetric group; this was proved in~\cite{TREILLIS}. (The MacNeille completion of a poset is a construction of the smallest lattice containing the poset as an induced subposet; see~\cite{trotter}. The Bruhat order is a well-studied partial order on the symmetric group which has a lovely generalization to any \emph{Coxeter group}; see~\cite{bjorner_brenti}.) 

We denote the poset of join irreducibles of this distributive lattice as $\aposet_n$, so that $J(\aposet_n)$ is in bijection with the set of $n \times n$ ASMs. In Definition~\ref{def:aposetex} we give an explicit construction of $\aposet_n$, and in Proposition~\ref{prop:oihf} we give an explicit bijection, which will be of use later, from the order ideals $J(\aposet_n)$ to height function matrices of order $n$.  
See~\cite{StrikerPoset} for further discussion of $\aposet_n$ and its broader family of \emph{tetrahedral posets}. Figure~\ref{ex:corresp} gives an ASM and the corresponding monotone triangle, corner sum matrix, height function matrix, and order ideal in $J(\aposet_4)$. Figure~\ref{ex:asm3} shows all the order ideals of $\aposet_3$. 

\begin{figure}[hbt]
\begin{center}
\scalebox{0.8}{
$\left(
\begin{array}{rrrr}
 0 & 0 & 1 & 0 \\
 1 & 0 & -1 & 1 \\
 0 & 0 & 1 & 0 \\
 0 & 1 & 0 & 0
\end{array}
\right) 
\begin{array}{ccccccc}
 & & & 3 & & & \\
 & & 1 & & 4 & & \\
 & 1 & & 3 & & 4 & \\
 1 & & 2 & & 3 & & 4
\end{array}
\left(
\begin{array}{cccc}
 0 & 0 & 1 & 1 \\
 1 & 1 & 1 & 2 \\
 1 & 1 & 2 & 3 \\
 1 & 2 & 3 & 4
\end{array}
\right) \left(
\begin{array}{ccccc}
 0 & 1 & 2 & 3 & 4\\
 1 & 2 & 3 & 2 & 3 \\
 2 & 1 & 2 & 3 & 2 \\
 3 & 2 & 3 & 2 & 1 \\
 4 & 3 & 2 & 1 & 0
\end{array}
\right) 
$
}
\scalebox{0.65}{
$\vcenter{\xymatrix @-1.2pc {
& & \circ & & & \circ & &  \circ & & \\
& \bullet \ar@{-}[ur] \ar@{-}[urrrr] & & \circ \ar@{-}[ul] \ar@{-}[urr]  & \ar@{-}[ur] \ar@{-}[urrr] \circ & & \ar@{-}[ul] \ar@{-}[ur] \circ & & \\
\bullet \ar@{-}[ur] \ar@{-}[urrrr] & & \bullet \ar@{-}[ur] \ar@{-}[ul] \ar@{-}[urr] \ar@{-}[urrrr] & & \bullet \ar@{-}[ul] \ar@{-}[urr] }}$}
\end{center}
\caption{A $4\times 4$ ASM and the corresponding monotone triangle, corner sum matrix, height function matrix, and order ideal in $J(\aposet_4)$.} 
\label{ex:corresp}
\end{figure}

\begin{figure}[hbt]
\begin{center} $\left \{ \begin{gathered}
\scalebox{0.6}{
$\xymatrix @-1.2pc {
& \circ & & & \circ \\
\ar@{-}[ur] \ar@{-}[urrrr] \circ & & \ar@{-}[ul] \ar@{-}[urr] \circ & & }$
$\xymatrix @-1.2pc {
& \circ & & & \bullet \\
\ar@{-}[ur] \ar@{-}[urrrr] \bullet & & \ar@{-}[ul] \ar@{-}[urr] \bullet & & }$
$\xymatrix @-1.2pc {
& \bullet & & & \circ \\
\ar@{-}[ur] \ar@{-}[urrrr] \bullet & & \ar@{-}[ul] \ar@{-}[urr] \bullet & & }$
$\xymatrix @-1.2pc {
& \circ & & & \circ \\
\ar@{-}[ur] \ar@{-}[urrrr] \bullet & & \ar@{-}[ul] \ar@{-}[urr] \circ & & }$
$\xymatrix @-1.2pc {
& \circ & & & \circ \\
\ar@{-}[ur] \ar@{-}[urrrr] \circ & & \ar@{-}[ul] \ar@{-}[urr] \bullet & & }$
$\xymatrix @-1.2pc {
& \bullet & & & \bullet \\
\ar@{-}[ur] \ar@{-}[urrrr] \bullet & & \ar@{-}[ul] \ar@{-}[urr] \bullet & & }$
$\xymatrix @-1.2pc {
& \circ & & & \circ \\
\ar@{-}[ur] \ar@{-}[urrrr] \bullet & & \ar@{-}[ul] \ar@{-}[urr] \bullet & & }$} \end{gathered} \right \}.$
\end{center}
\caption{The seven order ideals in $J(\aposet_3)$.}
\label{ex:asm3}
\end{figure}

We give here an explicit construction of $\aposet_n$ that we will use throughout the paper; see Figure~\ref{ex:aposet}. See Definition 8.4 of~\cite{prorow} for a different, but equivalent, explicit construction of $\aposet_n$ as a layering of successively smaller \emph{Type A positive root posets}. See also the constructions in~\cite{ELKP_DOMINO1}, \cite{ProppManyFaces}, and \cite{StrikerPoset}.
\begin{definition}
\label{def:aposetex}
Define the poset elements $\aposet_n$ as the coordinates $(i,j,k)$ in $\mathbb{Z}^3$ such that $0\leq i\leq n-2$, $0\leq j\leq n-2-i$, and $0\leq k\leq n-2-i-j$. Define the partial order via the following covering relations: $(i,j,k)$ covers $(i,j+1,k)$, $(i,j+1,k-1)$, $(i+1,j,k)$, and $(i+1,j,k-1)$, whenever these coordinates are poset elements. 
\end{definition}

\begin{figure}[hbt]
\begin{center}
\includegraphics[scale=0.8]{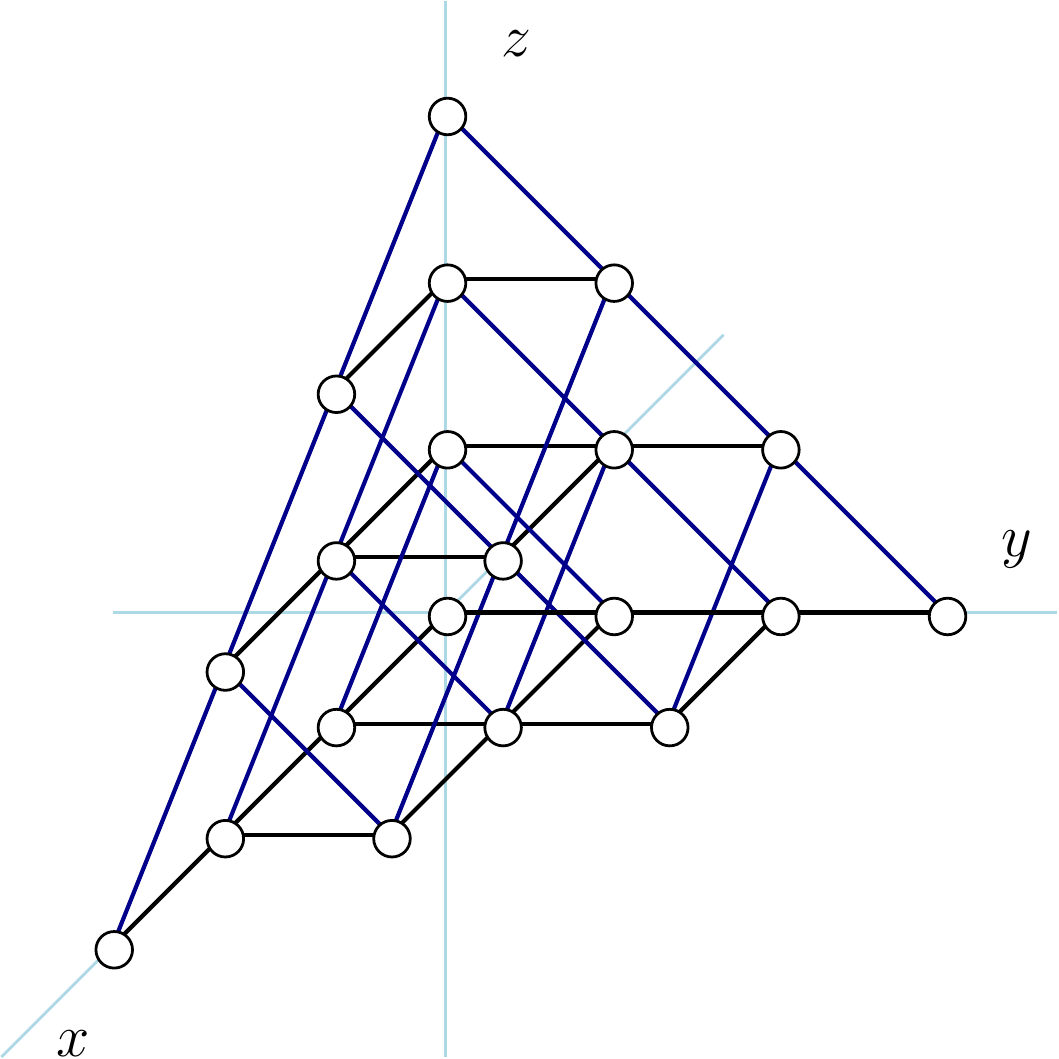}
\end{center}
\caption{The explicit construction of Definition~\ref{def:aposetex} for $\aposet_5$.}
\label{ex:aposet}
\end{figure}

\begin{proposition}
\label{prop:aposetranked}
$\aposet_n$ is a ranked poset in which the rank of $(i,j,k)$ in $\aposet_n$ equals $n-2-i-j$. The minimal elements of $\aposet_n$ are of the form $(i,n-2-i,0)$ and have rank $0$, and the maximal elements are of the form $(0,0,k)$ have rank $n-2$.
\end{proposition}

We give here a concrete bijection from order ideals $J(\aposet_n)$ to height function matrices; see Figure~\ref{ex:aposetgyr}. See also~\cite{ProppManyFaces} and \cite{prorow}.

\begin{proposition}
\label{prop:oihf}
There is an explicit bijection between the order ideals $J(\aposet_n)$ and height function matrices of order $n$.
\end{proposition}

\begin{proof}
We start with an order ideal in $J(\aposet_n)$. For $1\leq i,j\leq n-1$, define the subset $S_{i,j}:=\{(i-1-t,j-1-t,t) \}$ for any $t$ such that $(i-1-t,j-1-t,t)$ is an element of $\aposet_n$. (That is, $S_{i,j}$ is the intersection of the elements of $\aposet_n$ from the explicit construction in Definition~\ref{def:aposetex} with the line $(x,y,z)=(i-1-t,j-1-t,t)$.) See Figure~\ref{ex:aposetgyr}.  

Height function matrix entry $h_{i,j}$ for $1\leq i,j\leq n-1$ is determined by the cardinality of the intersection between the order ideal and $S_{i,j}$. If this cardinality equals $\ell$, then $h_{i,j}=\min(i+j,2n-i-j)-2\ell$, that is, $2\ell$ less than the maximum possible value of that height function entry. Note that $S_{i,j}$ will be a chain in $\aposet_n$ with elements from ranks of only one parity. This will be important in Corollary~\ref{cor:asmcor}. Also note that toggling an element out of an order ideal in $\aposet_n$ corresponds to increasing the corresponding height function matrix entry by two, whereas toggling an element into the order ideal corresponds to decreasing a height function matrix entry by two.
\end{proof}

\begin{figure}[htb]
\begin{center}
\includegraphics[scale=0.75]{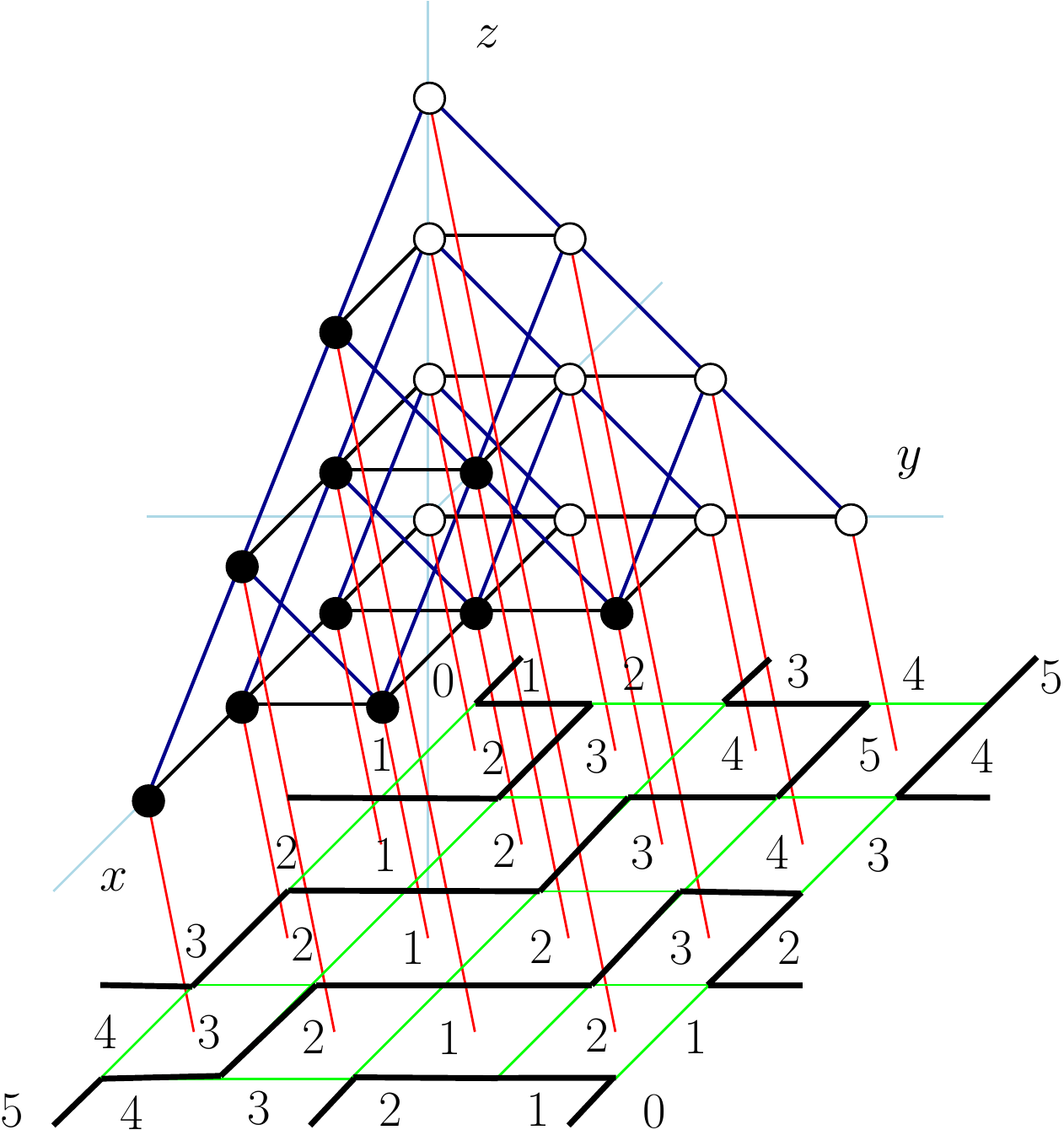}
\end{center}
\caption{An order ideal in $\aposet_5$ and its corresponding height function matrix and fully-packed loop configuration. The diagonal lines show the correspondence between the order ideal elements in the subset $S_{i,j}$ and the  height function matrix entry $h_{i,j}$.}
\label{ex:aposetgyr}
\end{figure}

We now have a bijection from alternating sign matrices through height function matrices to order ideals in a poset. Next, we define fully-packed loop configurations and relate them to height function matrices. See~\cite{ProppManyFaces} and~\cite{wieland} for further discussion.
\begin{definition} 
\label{def:fpl}
	Consider an $[n] \times [n]$ grid of dots in $\mathbb{Z}^2$. Beginning with the dot at the upper left corner, draw an edge from that dot up one unit. Then go around the grid, drawing such an external edge at every second dot (counting corner dots twice since at the corners, external edges could go in either of two directions). A \emph{fully-packed loop configuration (FPL)} of order $n$ is a set of paths and loops on the $[n] \times [n]$ grid with boundary conditions as described above, such that each of the $n^2$ vertices within the grid has exactly two incident edges.
\end{definition}

\begin{figure}[htbp]
\begin{center}
\includegraphics[scale=1]{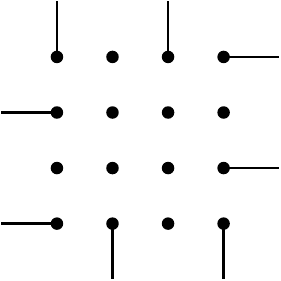}
\hspace{.75in}
\includegraphics[scale=1]{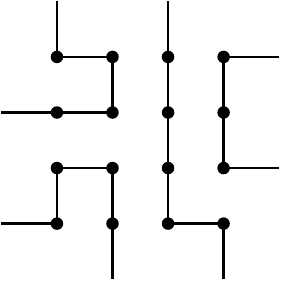}
\hspace{.75in}
\includegraphics[scale=1]{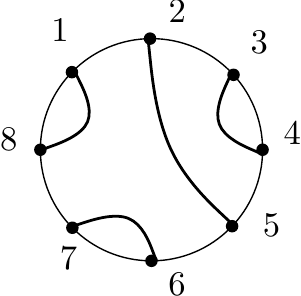}
\end{center}
\caption{Left: Fully-packed loop external edges. Center: The FPL corresponding to the ASM from Figure~\ref{ex:corresp}. Right: The link pattern corresponding to the FPL.}
\label{ex:fplbdry}
\end{figure}

Figure~\ref{ex:fplbdry}, left, gives an example of the FPL external edges, and Figure~\ref{ex:fplbdry}, center, shows the FPL in bijection with the ASM from Figure~\ref{ex:corresp}. Figure~\ref{ex:fplgyration} gives the FPLs of order $3$.

\begin{figure}[htb]
\begin{center}
\includegraphics[scale=0.85]{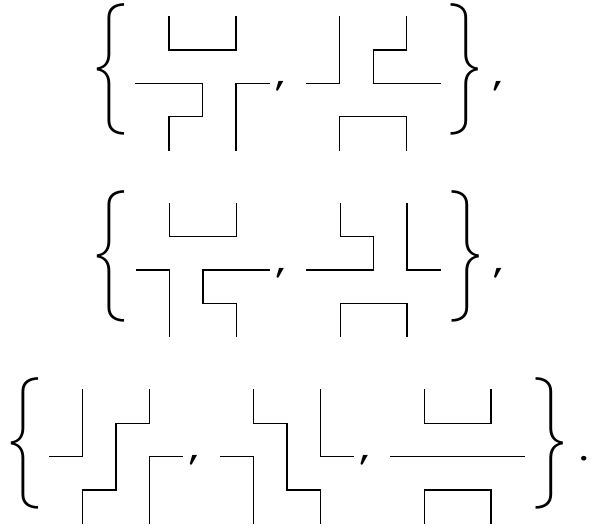}
\end{center}
\caption{The seven FPLs of order $3$.  They break into three orbits under gyration.}
\label{ex:fplgyration}
\end{figure}

We now construct a bijection between fully-packed loop configurations and height function matrices which will make more transparent the connection to toggling. As before, see also~\cite{ProppManyFaces} and \cite{prorow}.

\begin{proposition}
\label{prop:fplhf}
There is an explicit bijection between height function matrices of order $n$ and fully-packed loop configurations of order $n$.
\end{proposition}

\begin{proof}
Starting with a height function matrix, first draw an $[n]\times[n]$ grid of dots so that each interior number in the height function matrix has four surrounding corner dots. Separate two horizontally adjacent numbers by an edge if the numbers are $2k$ and $2k+1$ (in either order) for any integer $k$. Separate two vertically adjacent numbers by an edge if the numbers are $2k-1$ and $2k$ (in either order) for any integer $k$. This yields the required boundary conditions and the condition that interior vertices have exactly two adjacent edges.
See Figure~\ref{ex:hf_fpl} for a construction of this bijection.
\end{proof}

\begin{figure}[htb]
\begin{center}
\includegraphics[scale=0.85]{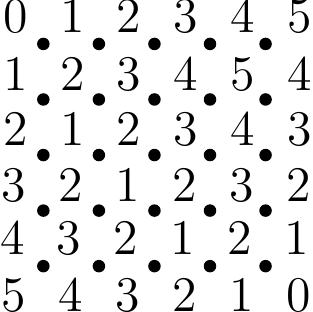}
\hspace{1cm}
\includegraphics[scale=0.85]{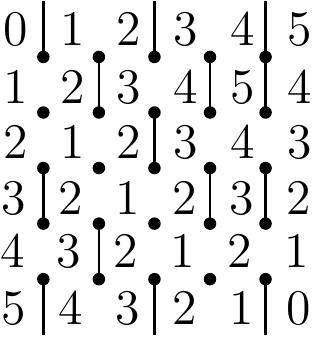}
\hspace{1cm}
\includegraphics[scale=0.85]{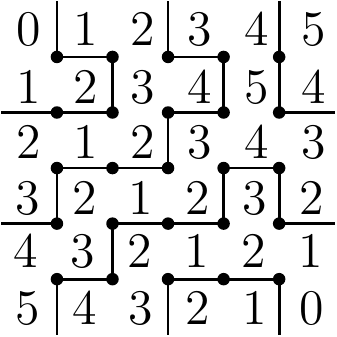}
\end{center}
\caption{The construction of an FPL from the height function matrix of Figure~\ref{ex:aposetgyr}. Left: The $5\times 5$ grid of FPL interior vertices overlaid on the height function matrix. Center: The addition of the vertical edges separating $2k$ and $2k+1$. Right: The addition of the horizontal edges separating $2k-1$ and $2k$.}
\label{ex:hf_fpl}
\end{figure}

We now define the \emph{link pattern} of an FPL; see Figure~\ref{ex:fplbdry}, right.
\begin{definition} 
\label{def:lp}
	Given a fully-packed loop, number the external edges clockwise, starting with the upper left external edge. Each external edge will be connected by a path to another external edge, and these paths will never cross.   
	This matching on the external edges will thus be a noncrossing matching on $2n$ points, and is called the \emph{link pattern} of the FPL.
\end{definition}

Observe the link pattern in Figure~\ref{ex:fplbdry}, right, is the noncrossing matching $(1,8)$, $(2,5)$, $(3,4)$, $(6,7)$. Note that the number of noncrossing matchings on $2n$ points is given by the $n$th Catalan number $Cat(n):=\frac{1}{n+1}\binom{2n}{n}$; see~\cite{Stanley_Cat} for more on the Catalan numbers.

In 2000, B.\ Wieland showed the surprisingly cyclic nature of an action called \emph{gyration} on fully-packed loop configurations.

\begin{definition}[\cite{wieland}]
\label{def:fplgyr}
Given an $[n]\times[n]$ grid of dots, color the interiors of the squares in a checkerboard pattern.  Given an FPL of order $n$ drawn on this grid, its \emph{gyration} is computed by first visiting all squares of one color then all squares of the other color,
performing the ``local move'' which swaps the edges around a square if the edges are parallel and otherwise leaves them fixed.
\end{definition}

Figure~\ref{ex:fpltoggle} shows the nontrivial local move of Wieland's gyration. Figure~\ref{ex:fplgyration} lists the FPLs by orbits under gyration.  
%Note the surprising cyclic symmetry, proved by Wieland in \cite{wieland}. 

\begin{theorem}[\cite{wieland}]
Gyration on an FPL rotates the link pattern by a factor of $2\pi/2n$.
\end{theorem} 

The following proposition, whose proof follows from the bijections of Propositions~\ref{prop:oihf} and~\ref{prop:fplhf}, states what the FPL local move corresponds to on height functions and order ideals.

\begin{proposition}
\label{prop:togcorresp}
Let $(h_{i,j})_{i,j=0}^{n}$ be a height function matrix, $I\in J(\aposet_n)$ the corresponding order ideal, and $F$ the corresponding FPL, via the bijections of Propositions~\ref{prop:oihf} and~\ref{prop:fplhf}.
Then the following are equivalent.
\begin{enumerate}
\item The FPL local move from Definition~\ref{def:fplgyr} applied to $F$ at the square in row $i$ and column $j$, where the rows (columns) are numbered from the top (left) starting at 1.
\item Incrementing or decrementing height function matrix entry $h_{i,j}$ by 2, if possible.
\item Toggling at some element in $S_{i,j}$ (from the proof of Proposition~\ref{prop:oihf}) such that incrementing $h_{i,j}$ corresponds to toggling an element out of the order ideal and decrementing $h_{i,j}$ corresponds to toggling an element into the order ideal.
%$\ell = \frac{1}{2}(\min(i+j,2n-i-j)- h_{i,j})$
\end{enumerate}
\end{proposition}

\begin{figure}[htb]
\begin{center}
\includegraphics[scale=0.85]{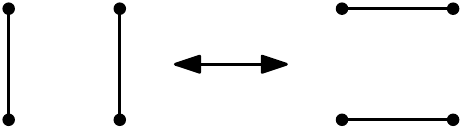}
\end{center}
\caption{Gyration's nontrivial local move. This corresponds to both a nontrivial toggle in $T(\aposet_n)$ and the incrementing or decrementing of a height function matrix entry by $2$.}
\label{ex:fpltoggle}
\end{figure}

In Proposition~\ref{prop:asmgyr}, we give the characterization from~\cite{prorow} of gyration as a toggle group element, namely, that gyration on fully-packed loop configurations of order $n$ is equivalent to the toggle group action on $J(\aposet_n)$ which toggles the elements in even ranks first, then odd ranks. Then in Corollary~\ref{cor:asmcor}, we apply this characterization to a component of Cantini and Sportiello's proof~\cite{razstrogpf} of the Razumov-Stroganov correspondence.

But before we discuss the proof, we need to explain the statement of the Razumov-Stroganov correspondence.

\section{The Razumov-Stroganov correspondence}
\label{sec:rs}
The Razumov-Stroganov correspondence relates the ground state eigenvector of the $O(1)$ dense loop model on a semi-infinite cylinder to the vector of the distribution of fully-packed loops by link pattern. Though this correspondence provides an important link between physics and combinatorics, the result can be stated in purely combinatorial terms.

We first define the rotation operator $R$ and the \emph{(affine) Temperley-Lieb operators} 
$e_j, 1\leq j\leq 2n$, which act on link patterns on $2n$ vertices. (In Definition~\ref{def:lp}, we defined the link pattern of a fully-packed loop configuration. We now consider link patterns abstractly as noncrossing matchings on $2n$ vertices; see Figure~\ref{ex:fpl_TL}.)
Let $R$ be the operator that rotates the link pattern one step counterclockwise. Let $e_j$ be the operator that acts as the identity if $j$ and $j+1$ are connected, and otherwise connects $j$ with $j+1$ and matches the indices previously matched to $j$ and $j + 1$ with each other (consider the indices modulo $2n$ so that $e_{2n}$ matches $2n$ and $1$). 
These operators, along with some relations, generate the (affine) Temperley-Lieb algebra, which can be nicely visualized as a \emph{diagram algebra}. See~\cite{razstrogpf} and~\cite{de_Gier_loops} for further explanation.

\begin{figure}[htb]
\begin{center}
\includegraphics[scale=0.85]{fpl_ex_linkpattern.pdf}
\hspace{.75in}
\includegraphics[scale=0.85]{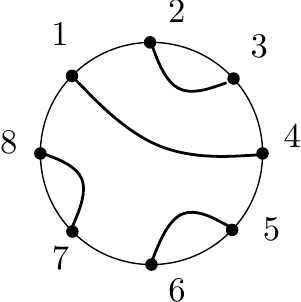}
\hspace{.75in}
\includegraphics[scale=0.85]{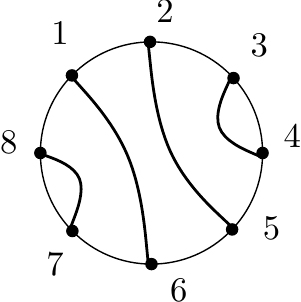}
\end{center}
\caption{Left: the link pattern of Figure~\ref{ex:fplbdry}. Center: the action of $R$ on this link pattern. Right: the action of $e_7$ on the link pattern from the left.}
\label{ex:fpl_TL}
\end{figure}

Using the Temperley-Lieb operators, we give a combinatorial statement of the 
\linebreak Razumov-Stroganov correspondence, as first formulated by Razumov and Stroganov in Conjecture 2 of~\cite{razstrog}. 
\begin{theorem}[Cantini-Sportiello, 2011]
\label{thm:razstrogcomb}
The average value, over all $2n$ Temperley-Lieb operators $e_j$, of the number of FPLs with link pattern $\pi'$ such that $e_j \pi'=\pi$ equals the number of FPLs with link pattern $\pi$.
\end{theorem}

Stated another way, the number of FPLs whose link pattern is mapped to $\pi$ by any of the $e_j$ equals $2n$ times the number of FPLs with link pattern $\pi$. 

We now state the Razumov-Stroganov correspondence from the perspective (and using the notation) of statistical physics; this was given as an alternative formulation of the conjecture in~\cite{razstrog} and is the statement proved in~\cite{razstrogpf}. The operator $H_n:=\sum_{j=1}^n e_j$ is called the \emph{Hamiltonian}. If we define $|s_n\rangle$ to be the vector of Catalan length $Cat(n)$ which gives the distribution of FPLs by link pattern, we can then write the Razumov-Stroganov correspondence as:
\[ H_n |s_n\rangle = 2n |s_n\rangle.\]
Or, defining $\mathbf{RS}_n:=(H_n - 2n)|s_n\rangle$, we can restate this as $\mathbf{RS}_n=\mathbf{0}$.

The above explanation is sufficient for the purposes of understanding the combinatorial side of the Razumov-Stroganov correspondence. But in order to understand the connection to the $O(1)$ dense loop model on a semi-infinite cylinder, we give the following sketch of an explanation. See Section 3.3 of~\cite{ZinnJustin} or Section 1.4 of~\cite{Romik2014} for full details. 

Define the $O(1)$ dense loop model on a semi-infinite cylinder as follows. Fix a semi-infinite cylinder of circumference $2n$. Within each unit square, place one of the following two tiles: \includegraphics[scale=.03]{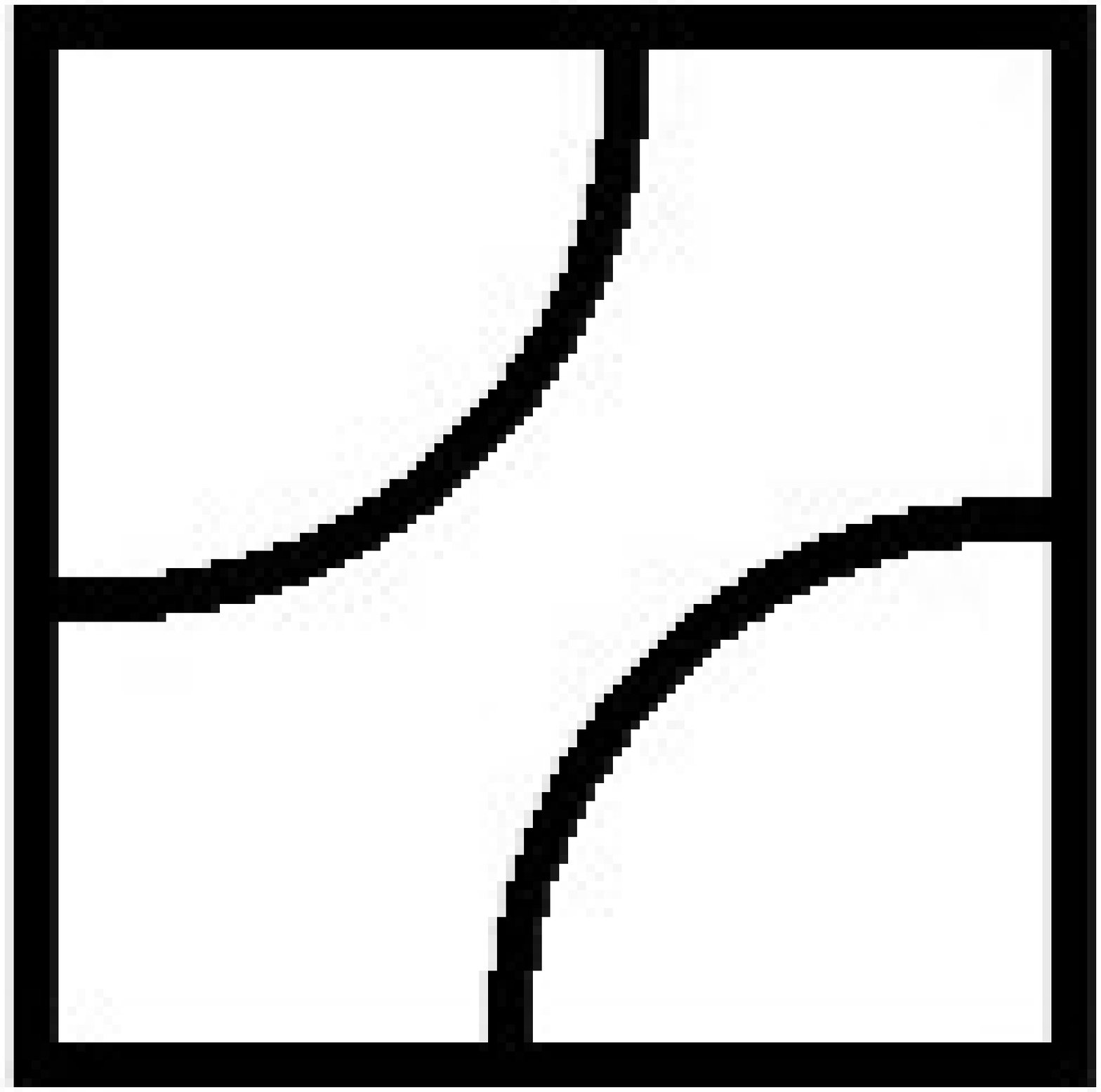} or \includegraphics[scale=.03]{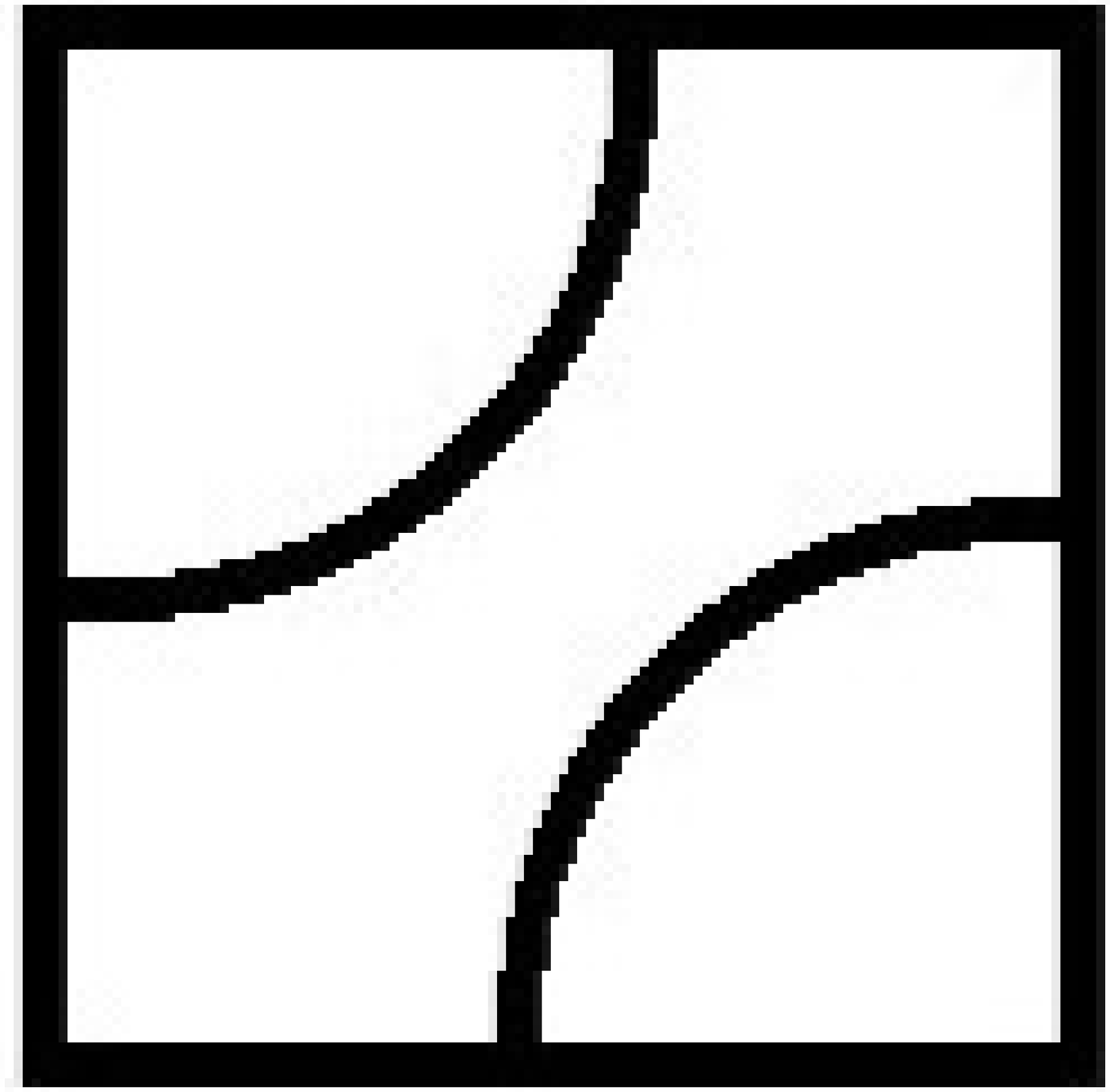}, where each is chosen with probability $\frac{1}{2}$. Note that, almost certainly, each path that begins at one of the $2n$ points on the finite end of the cylinder will terminate at another of these endpoints, rather than going off to infinity; see Lemma 1.6 of~\cite{Romik2014}. So each configuration of the cylinder induces a link pattern (or noncrossing matching) on the $2n$ endpoints. 

An alternative statement of the Razumov-Stroganov conjecture is the following. Let $|\tilde{\Psi}\rangle$ be the vector of length $Cat(n)$ that gives the probability distribution by link pattern of configurations of the $O(1)$ dense loop model on a semi-infinite cylinder of circumference $2n$ (this is called the \emph{ground state eigenvector}). Then $|\tilde{\Psi}\rangle = \frac{1}{ASM(n)}|s_n\rangle$ where $ASM(n)$ is the number of $n\times n$ alternating sign matrices (or fully-packed loops of order $n$). We get from this statement to the statement of Theorem~\ref{thm:razstrogcomb} as follows.

Consider the Markov process defined by adding a row to the cylinder with each of the $2n$ new tiles chosen from the two possibilities with probabilities $p$ and $1-p$. It can be shown, using either the Yang-Baxter equation (see, for example, Section 4 of~\cite{PeledRomik}) or a bijective approach (see Section 3 of~\cite{PeledRomik}) that the associated transfer matrices with different values of $p$ commute, so that, surprisingly, the probability distribution by link pattern is independent of $p$. So we can consider, rather, the limit $p\to 0$. (See, for example, Theorem 1.8 and Appendix B of~\cite{Romik2014}.) Thus, a new row of the cylinder usually consists of all one type of tile; this serves to rotate the link pattern by the action of the rotation operator $R$. But, occasionally, there will be a single tile of the other type in the new row. In this situation, the effect on the link pattern amounts to the rotation $R$ combined with the action of a Temperley-Lieb operator $e_j$. Thus, the action on the link pattern of adding a row in the Markov process is given by $R$ or $R e_j$ for some $1\leq j\leq 2n$. 
So the probability distribution by link pattern of configurations of the $O(1)$ dense loop model on a semi-infinite cylinder is equal to the probability distribution by link pattern of the Temperley-Lieb Markov process. 

\section{Homomesy}
\label{sec:homomesy}

In 2013, J.\ Propp and T.\ Roby isolated and named the \emph{homomesy} phenomenon (Greek for ``same middle'')~\cite{homomesy}.  This phenomenon occurs when given a set of combinatorial objects, a bijective action on this set, and a combinatorial statistic associated to each object in the set, the average value of this statistic over an orbit of the action is the same no matter which orbit you consider. That is, this average on an orbit equals the global average value of the statistic. We state this more formally in the following definition, in slightly less generality than in \cite{homomesy}.

\begin{definition}
Given a finite set $X$ of combinatorial objects, an invertible map $\tau$ from $X$ to
itself, and a \emph{combinatorial statistic} on $X$, that is, a map $f : X \rightarrow \mathbb{Z}$, we say the triple $(X, \tau, f)$ exhibits \emph{homomesy} if and only if there exists a constant $c\in \mathbb{Q}$ such
that for every $\tau$-orbit $\mathcal{O}\subseteq X$ 
\[\frac{1}{|\mathcal{O}|}\sum_{x\in\mathcal{O}}f(x)=c.\]
In this situation we say that the statistic $f$ is \emph{homomesic} (or $c$-\emph{mesic}) relative to the action of $\tau$ on $X$. 
\end{definition}

In the same paper~\cite{homomesy}, Propp and Roby proved that homomesy occurs in the case where the set $X$ equals the set of order ideals of the product of two chains poset, the action $\tau$ is rowmotion, and the statistic $f$ is either the size of the order ideal or the size of the corresponding antichain of maximal elements. They also proved that the size of the order ideal statistic is homomesic under \emph{promotion}, a toggle group action shown in~\cite{prorow} to be conjugate to rowmotion. Homomesy has been a growing area of research, with further results found on other posets~\cite{SHaddadan}~\cite{Rush}, semistandard Young tableaux~\cite{bloom_homomesy}, oscillating tableaux~\cite{HopkinsZhang}, and piecewise-linear generalizations of toggle group actions~\cite{EinsteinPropp}.

In the next section, we prove that rowmotion and another of its conjugates exhibits homomesy with respect to a statistic isolating the \emph{toggleability} of a poset element. We then apply this result to the Razumov-Stroganov correspondence.

\section{Results}
\label{sec:results}

In this section, we first define a statistic on order ideals which isolates the \emph{toggleability} of an element in the poset. 
We then show that this statistic is homomesic with respect to both rowmotion (Lemma~\ref{lem:rowhomomesy}) and $\gyr$ (Theorem~\ref{thm:genposet}), where $\gyr$ is the toggle group action given in Definition~\ref{def:gyr} which corresponds to FPL gyration.  Finally, we obtain a key lemma from Cantini and Sportiello's proof of the Razumov-Stroganov correspondence~\cite{razstrogpf}, as a corollary of Theorem~\ref{thm:genposet}.

\begin{definition}
\label{def:toggleability}
Fix a finite poset $P$. For each $p \in P$, define the \emph{toggleability} statistic $\togstat_p: J(P) \to \{-1,0,1\}$ as follows:
\begin{center}
$\togstat_p(I) = \left\{
	\begin{array}{rl}
		1 & \text{ if } p\in (P\setminus I)_{\min},\\
		-1 & \text{ if } p \in I_{\max},\\
		0 & \text{ otherwise}.\\
	\end{array} \right.
$
\end{center}
\end{definition}

Given an order ideal $I$, the poset elements $p$ for which $\togstat_p(I)\neq 0$ are called the \emph{toggleable} elements of $P$ for $I$, since these are precisely the elements of $P$ for which the toggle $t_p$ acts nontrivially on $I$. 

\begin{lemma}
\label{lem:rowhomomesy}
Given any finite poset $P$ and $p\in P$, $\togstat_p$ is 0-mesic with respect to rowmotion.
\end{lemma}
\begin{proof}

Consider  $I\in J(P)$.  Then it follows immediately from Definition~\ref{def:row} (for rowmotion) and Definition~\ref{def:toggleability} (for the toggleability statistic) that the four statements $\togstat_p(I)=1$, $p\in(P\setminus I)_{\min}$, $p\in(\row(I))_{\max}$ and $\togstat_p(\row(I))=-1$ are all equivalent.  Therefore, the only possibilities for $(\togstat_p(I),\togstat_p(\row(I)))$ are $(1,-1)$, $(0,1)$, $(0,0)$, $(-1,1)$ and $(-1,0)$. Hence, along any orbit of $\row$, the nonzero values of $\togstat_p$ alternate in sign (with no 0's separating a $1$ followed by a $-1$, but potentially arbitrarily many 0's separating a $-1$ followed by a $1$).  
\end{proof}

It is not true that we automatically get the same result for all toggle group actions conjugate to rowmotion. For example, in the poset $P=\{a,b,c\}$ with relations $a<c, b<c$, the toggleability statistic $\togstat_c$ is not homomesic with respect to \emph{promotion} (toggle in the order $a,c,b$). But for the particular conjugate action of Definition~\ref{def:gyr} below, we prove in Theorem~\ref{thm:genposet} homomesy of the same statistic. 

\begin{definition} \rm
\label{def:gyr}
For any finite ranked poset $P$, define \emph{gyration} $\gyr:J(P)\rightarrow J(P)$ as the toggle group action which toggles the elements in even ranks first, then odd ranks. 
\end{definition}

Recall from Lemma~\ref{lem:commute} that $t_p$ and  $t_q$ commute whenever there is no cover relation between $p$ and $q$. Thus, all elements in ranks of the same parity commute, so the above definition is unambiguous.
Also, note that for any permutation of the ranks of a ranked poset, toggling the elements by rank according to that permutation is conjugate to rowmotion in the toggle group (this is Lemma 2 of~\cite{fonderflaass}).  So we have the following.
\begin{proposition}
For any finite ranked poset $P$, there is an equivariant bijection between $J(P)$ under $\row$ and under $\gyr$, that is, $\row$ and $\gyr$ are conjugate elements in the toggle group $T(P)$.
\end{proposition}

In the alternating sign matrix case $P=\aposet_n$, we have the following proposition, which follows directly from Propositions~\ref{prop:togcorresp} and~\ref{prop:oihf}.
 
\begin{proposition}[Proposition 8.12 in~\cite{prorow}]
\label{prop:asmgyr}
$\gyr$ acting on $J(\aposet_n)$ is equivalent to gyration on fully-packed loops of order $n$. 
\end{proposition}

In~\cite{prorow}, N.\ Williams and the author used these two propositions to show the following theorem. 

\begin{theorem}[Theorem 8.13  in~\cite{prorow}]
\label{thm:asmgyr}
$J(\aposet_n)$ under $\row$ and fully-packed loop configurations of order $n$ under gyration  are in equivariant bijection.  
\end{theorem}

Given the characterization from Proposition~\ref{prop:asmgyr} of Wieland's gyration on fully-packed loop configurations in terms of the toggle group action $\gyr$, we now prove the following theorem about the toggleability statistic $\togstat_p$ on any ranked poset. We then specialize to alternating sign matrices in Corollary~\ref{cor:asmcor}.
\begin{theorem}
\label{thm:genposet}
Given any finite ranked poset $P$ and $p\in P$, $\togstat_p$ is 0-mesic with respect to $\gyr$.
\end{theorem}

\begin{proof}
Fix any $p\in P$. We claim that as you go through an orbit of $\gyr$, the nonzero values of $\togstat_p$ alternate in sign. Suppose $p$ is in a rank of even parity. (If $p$ is in a rank of odd parity, the argument follows similarly by considering the action of $\gyr^{-1}$.)

\emph{Case 1:} If $\togstat_p(I)=-1$, then when the even ranks of $P$ are toggled, $p$ is removed from $I$. Then in the odd toggle step, if no element $p$ covers is removed, then $\togstat_p(\gyr(I))=1$ (\emph{Case 2}), otherwise, $\togstat_p(\gyr(I))=0$ and $p\notin \gyr(I)$ (\emph{Case 4}). 

\emph{Case 2:} If $\togstat_p(I)=1$, then when the even ranks of $P$ are toggled, $p$ is added to the order ideal. Then in the odd rank toggle step, if no element covering $p$ is added, then $\togstat_p(\gyr(I))=-1$ (\emph{Case 1}). Otherwise,  $\togstat_p(\gyr(I))=0$ and $p\in \gyr(I)$ (\emph{Case 3}). 

\emph{Case 3:} If $\togstat_p(I)=0$ and $p\in I$, then since $p$ is not toggleable, there must be some $q$ covering $p$ such that $q\in I$. So when the even ranks are toggled, $p$ remains in the order ideal. Thus $\togstat_p(\gyr(I))$ is either equal to $-1$ (\emph{Case 1}) or equals $0$ such that $\gyr(I)$ contains $p$ (\emph{Case 3}). 

\emph{Case 4:} If $\togstat_p(I)=0$ and $p\notin I$, then since $p$ is not toggleable, there must be some $q$ covered by $p$ such that $q\notin I$. So when the elements in even ranks are toggled, $p$ continues to not be in the order ideal. Thus $\togstat_p(\gyr(I))$ cannot equal $-1$ and either equals $1$ (\emph{Case 2}) or $0$ and $p\notin \gyr(I)$ (\emph{Case~4}).

Thus, the action of $\gyr$ is as in the following graphic.
\begin{center}
\includegraphics[scale=.75]{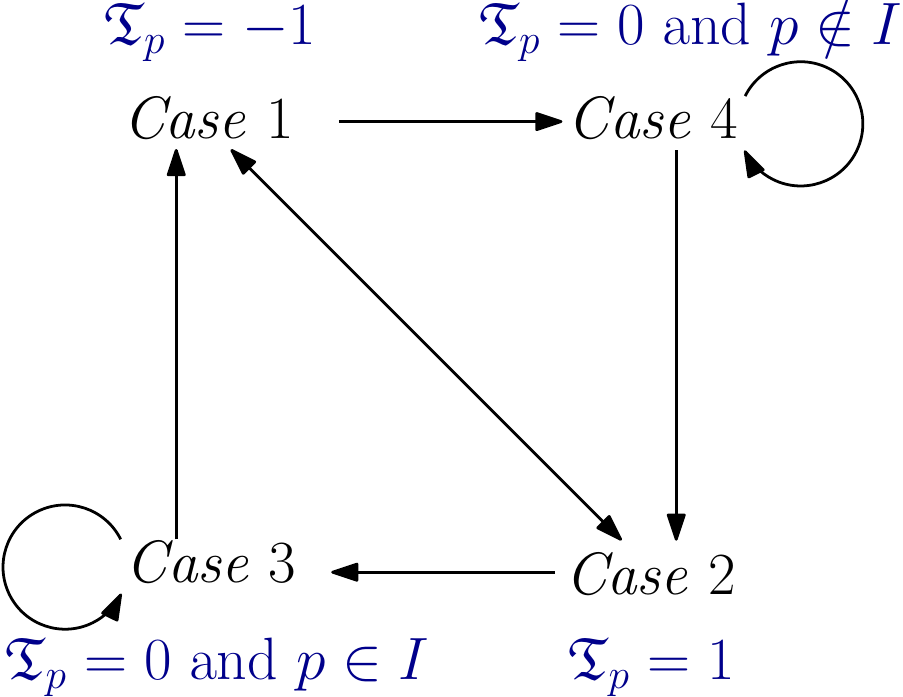}
\end{center}
Therefore, the nonzero values of $\togstat_p$ alternate in sign throughout an orbit of $\gyr$.
\end{proof}

Lemma 4.1 from Cantini and Sportiello's 2011 proof of the Razumov-Stroganov conjecture follows from this theorem, when we specialize to the ASM poset. We have listed a restatement of it here as a corollary.
\begin{corollary}[Lemma 4.1 from \cite{razstrogpf}]
\label{cor:asmcor}
Fix any square $\alpha$ in the $[n]\times [n]$ grid. Then the number of FPLs in an orbit of gyration with edge configuration $|\alpha|$ equals the number with configuration $\underline{\overline{\mbox{ }\alpha\mbox{ }}}$. 
\end{corollary}
\begin{proof}
As in Lemma 4.1 of Cantini-Sportiello, let $\mathcal{N}_{\alpha}$ equal $-1$ if the FPL edge configuration around $\alpha$ is $\underline{\overline{\mbox{ }\alpha\mbox{ }}}$, $+1$ if it is $|\alpha|$, and 0 otherwise. Via the bijections of Propositions~\ref{prop:oihf} and~\ref{prop:fplhf}, the square $\alpha$ in row $i$ column $j$ in the grid corresponds to the chain $S_{i,j}=\{(i-1-t,j-1-t,t)\}\subseteq \aposet_n$; see Figure~\ref{ex:aposetgyr}. We noted in the proof of Proposition~\ref{prop:oihf} that the chain $S_{i,j}$ consists of poset elements all of which are in ranks of the same parity. Since $S_{i,j}$ is a chain of elements in which no two have a covering relation, if we fix any order ideal $I$, there will be at most one element $p\in S_{i,j}$ such that $\togstat_p(I)\neq 0$. By Proposition~\ref{prop:togcorresp}, summing the statistics on all the poset elements in $S_{i,j}$ gives the corresponding statistic on the square $\alpha$, that is, for any $I\in J(P)$, $\sum_{p\in S_{i,j}} \togstat_p(I) = \mathcal{N}_{\alpha}(I)$. Since each $\togstat_p$ is 0-mesic on orbits of $\gyr$, $\mathcal{N}_{\alpha}$ is 0-mesic on orbits of gyration. Thus the statement follows from the definition of $\mathcal{N}_{\alpha}$. 
\end{proof}
Note that Corollary~\ref{cor:asmcor} is actually a stronger statement, as the toggleability statistic $\togstat_p$ is a refinement of the statistic $\mathcal{N}_{\alpha}$ from \cite{razstrogpf}.

\section*{Acknowledgments}
The author thanks Roger Behrend and the anonymous referee for many helpful comments.

\end{document}